\documentclass[reqno]{amsart}
\usepackage{amsmath}
\usepackage{amsfonts}
\usepackage{graphics}
\usepackage{booktabs}
\usepackage{subfigure}
\usepackage{graphicx}

\newtheorem{theorem}{Theorem}
\theoremstyle{plain}

\newtheorem{corollary}{Corollary}

\newtheorem{lemma}{Lemma}

\newtheorem{proposition}{Proposition}
\newtheorem{remark}{Remark}

\numberwithin{equation}{section}

\begin{document}
\title{Interiors of continuous images of the middle-third Cantor set}

\begin{abstract}
Let $C$ be the middle-third Cantor set, and $f$ a continuous function
defined on an open set $U\subset \mathbb{R}^{2}$. Denote the image
\begin{equation*}
f_{U}(C,C)=\{f(x,y):(x,y)\in (C\times C)\cap U\}.
\end{equation*}%
If $\partial _{x}f$, $\partial _{y}f$ are continuous on $U,$ and there is a
point $(x_{0},y_{0})\in (C\times C)\cap U$ such that
\begin{equation*}
1<\left\vert \frac{\partial _{x}f|_{(x_{0},y_{0})}}{\partial
_{y}f|_{(x_{0},y_{0})}}\right\vert <3\text{ or }1<\left\vert \frac{\partial
_{y}f|_{(x_{0},y_{0})}}{\partial _{x}f|_{(x_{0},y_{0})}}\right\vert <3,
\end{equation*}%
then $f_{U}(C,C)$ has a non-empty interior. As a consequence, if
\begin{equation*}
f(x,y)=x^{\alpha }y^{\beta }(\alpha \beta \neq 0),\text{ }x^{\alpha }\pm
y^{\alpha }(\alpha \neq 0)\text{ or }\sin (x)\cos (y),
\end{equation*}%
then $f_{U}(C,C)$ contains a non-empty interior.
\end{abstract}

\author{kan Jiang}
\address{Department of Mathematics, Ningbo University, Ningbo 315211, P. R.
China}
\email{jiangkan@nbu.edu.cn; kanjiangbunnik@yahoo.com}
\author{lifeng Xi}
\address{Department of Mathematics, Ningbo University, Ningbo 315211, P. R.
China}
\email{xilifeng@nbu.edu.cn; xilifengningbo@yahoo.com}
\thanks{Lifeng Xi is the corresponding author. The work is supported by
National Natural Science Foundation of China (Nos. 11831007, 11771226,
11701302, 11371329, 11471124, 11671147). The work is also supported by K.C.
Wong Magna Fund in Ningbo University.}
\maketitle

\section{Introduction}

The middle-third Cantor set, denoted by $C$, is an elegant set in set
theory. Usually, it is used to construct some counterexamples in analysis.
However, there are still many open problems for this set. For instance, the
multiplication on the Cantor sets \cite{Tyson, Shmerkin}, the sections of
the products of the Cantor sets \cite{FG}, and so forth.

One of the main motivations of this paper is due to a result of Hochman and
Shmerkin \cite{Hochman2012}: Let $K_{1}$ and $K_{2}$ be two self-similar
sets with IFS's $\{f_{i}(x)=r_{i}x+a_{i}\}_{i=1}^{n}$ and $%
\{g_{j}(x)=r_{j}^{\prime }x+b_{j}\}_{j=1}^{m}$ respectively, if there are
some $r_{i},r_{j}^{\prime }$ such that $\log |r_{i}|/\log |r_{j}^{\prime
}|\notin \mathbb{Q},$ then
\begin{equation*}
\dim _{H}(K_{1}+K_{2})=\min \{\dim _{H}(K_{1})+\dim _{H}(K_{2}),1\},
\end{equation*}%
where $K_{1}+K_{2}=\{x+y:x\in K_{1},y\in K_{2}\}$. The condition in the
above result is called the irrational condition. In \cite{Shmerkin},
Shmerkin stated
\begin{equation*}
\dim _{H}(K_{1}\cdot K_{2})=\min \{\dim _{H}(K_{1})+\dim _{H}(K_{2}),1\},
\end{equation*}
where $K_{1}\cdot K_{2}=\{xy:x\in K_{1},y\in K_{2}\}$. It is natural to
consider that without the irrational condition, how large $K_{1}\cdot K_{2}$
is in the sense of Hausdorff dimension or Lebesgue measure.

Let $f$ be a continuous function defined on an open set $U\subset \mathbb{R}%
^{2}$. Denote the image%
\begin{equation*}
f_{U}(C,C)=\{f(x,y):(x,y)\in (C\times C)\cap U\}.
\end{equation*}%
For convenience, we write $f(C,C)=f_{\mathbb{R}^{2}}(C,C).$ Steinhaus \cite%
{HS} proved that if $f(x,y)=x-y$, then $f(C,C)=[-1,1]$. As a result, $%
C+C=[0,2]$ since $C$ is symmetric at $1/2,$ i.e. $C=1-C.$ In \cite{Tyson},
Athreya, Reznick and Tyson considered the multiplication on $C$, and proved
that
\begin{equation*}
17/21\leq \mathcal{L}(C\cdot C)\leq 8/9,
\end{equation*}%
where $\mathcal{L}$ denotes the Lebesgue measure and $C\cdot C=\{xy:x,y\in
C\}.$ One can find more results in \cite{Hall, PS, XiKan1, XiKan2,SumKan}
for the arithmetic representations of real numbers.

Motivated by Athreya, Reznick and Tyson's result, it is natural to ask
whether $C\cdot C$ contains a non-empty interior. To the best of our
knowledge, we, up to now, cannot find an answer to this kind of question for
two general self-similar sets.

In fact, whether a fractal set contains a non-empty interior is a crucial
problem in fractal geometry and dynamical systems. This is the second reason
why we analyze the existence of the interior. Schief \cite{Schief}, Bandt
and Graf \cite{Bandt} showed the relation among the open set condition,
positive Hausdorff measure and non-empty interior. Dajani et al. \cite{DJK},
Hare and Sidorov \cite{Hare2017,Hare2016} found that the existence of the
non-empty interior of a class of self-affine sets is equivalent to the
existence of the simultaneous expansions. In dynamical system there is a
celebrated conjecture posed by Palis \cite{Palis}, i.e. whether it is true
(at least generically) that the arithmetic sum of dynamically defined Cantor
sets either has measure zero or contains a non-empty interior. This
conjecture was solved in \cite{Yoccoz}. However, for the general
self-similar sets this conjecture is still open.

In this paper, we prove the following result.

\begin{theorem}
\label{Main} If $\partial _{x}f$, $\partial _{y}f$ are continuous on $U,$
and there is a point $(x_{0},y_{0})\in (C\times C)\cap U$ such that
\begin{equation*}
1<\left\vert \frac{\partial _{x}f|_{(x_{0},y_{0})}}{\partial
_{y}f|_{(x_{0},y_{0})}}\right\vert <3\text{ or }1<\left\vert \frac{\partial
_{y}f|_{(x_{0},y_{0})}}{\partial _{x}f|_{(x_{0},y_{0})}}\right\vert <3.
\end{equation*}%
Then $f_{U}(C,C)$ contains a non-empty interior.
\end{theorem}

\begin{corollary}
\label{Cor} Let $C$ be the middle-third Cantor set. If $f(x,y)$ is one of
the following functions,
\begin{equation*}
x^{\alpha }y^{\beta }(\alpha \beta \neq 0),\text{ }x^{\alpha }\pm y^{\alpha
}(\alpha \neq 0),\text{ }x\pm y^{2},\text{ }\sin (x)\cos (y),
\end{equation*}%
then $f_{U}(C,C)$ contains a non-empty interior.
\end{corollary}

This paper is arranged as follows. In Section 2, we give a proof of Theorem %
\ref{Main}. In Section 3, we give some remarks.

\bigskip

\section{Proof of Theorem \protect\ref{Main}}

The middle-third Cantor set can be generated by an iterated function system,
i.e. $C$ is the unique non-empty compact set satisfying the equation:
\begin{equation*}
C=f_{1}(C)\cup f_{2}(C),
\end{equation*}%
where $f_{1}(x)=\dfrac{x}{3},f_{2}(x)=\dfrac{x+2}{3}$, see \cite{Hutchinson}%
. Given $J=[a,b],$ let
\begin{equation*}
\widetilde{J}=[a,a+\frac{b-a}{3}]\cup \lbrack b-\frac{b-a}{3},b].
\end{equation*}%
Let $H=[0,1]$. For any $(i_{1},\cdots ,i_{n})\in \{1,2\}^{n}$, we call $%
f_{i_{1},\cdots ,i_{n}}(H)=(f_{i_{1}}\circ \cdots \circ f_{i_{n}})(H)$ a
basic interval of rank $n,$ which has length $3^{-n}$.  We say that $I\times
J$ is a basic square of $C\times C$, if $I$ and $J$ are basic intervals of
the same rank. Denote by $H_{n}$ the collection of all these basic intervals
of rank $n$. Let $J\in H_{n}$, then $\widetilde{J}=\cup _{i=1}^{2}I_{n+1,i}$%
, where $I_{n+1,i}\in H_{n+1}$ and $I_{n+1,i}\subset J$ for $i=1,2$. Let $%
[A,B]\subset \lbrack 0,1]$, where $A$ and $B$ are the left and right
endpoints of some basic intervals in $H_{k}$ for some $k\geq 1$,
respectively. $A$ and $B$ may not be in the same basic interval. Let $F_{k}$
be the collection of all the basic intervals in $[A,B]$ with length $%
3^{-k},k\geq k_{0}$ for some $k_{0}\in \mathbb{N}^{+}$, i.e. the union of
all the elements of $F_{k}$ is denoted by $G_{k}=\cup _{i=1}^{t_{k}}I_{k,i}$%
, where $t_{k}\in \mathbb{N}^{+}$, $I_{k,i}\in H_{k}$ and $I_{k,i}\subset
\lbrack A,B]$. Clearly, by the definition of $G_{n}$, it follows that $%
G_{n+1}\subset G_{n}$ for any $n\geq k_{0}.$ Similarly, suppose that $M$ and
$N$ are the left and right endpoints of some basic intervals in $H_{k}$.
Denote by $G_{k}^{\prime }$ the union of all the basic intervals with length
$3^{-k}$ in the interval $[M,N]$, i.e. $G_{k}^{\prime }=\cup
_{j=1}^{t_{k}^{\prime }}J_{k,j}$, where $t_{k}^{\prime }\in \mathbb{N}^{+}$,
$J_{k,j}\in H_{k}$ and $J_{k,j}\subset \lbrack M,N]$.

The following Lemma \ref{key1} comes from \cite{Tyson} and \cite{XiKan1},
here we give its proof just for the self-containedness of the paper.

\begin{lemma}
\label{key1} Let $F:U\rightarrow \mathbb{R}$ be a continuous function$.$
Suppose $A$ and $B$ ($M$ and $N$) are the left and right endpoints of some
basic intervals in $H_{k_{0}}$ for some $k_{0}\geq 1$ respectively such that
$[A,B]\times \lbrack M,N]\subset U.$ Then $C\cap \lbrack A,B]=\cap _{n={k_{0}%
}}^{\infty }G_{n}$, and $C\cap \lbrack M,N]=\cap _{n={k_{0}}}^{\infty
}G_{n}^{\prime }$. Moreover, if for any $n\geq k_{0}$ and any two basic
intervals $I\subset G_{n}$, $J\subset G_{n}^{\prime }$ such that
\begin{equation*}
F(I,J)=F(\widetilde{I},\widetilde{J}),
\end{equation*}%
then $F(C\cap \lbrack A,B],C\cap \lbrack M,N])=F(G_{k_{0}},G_{k_{0}}^{\prime
}).$
\end{lemma}

\begin{remark}
The lemma is also valid when we replace $C\cap \lbrack A,B]$ (or $C\cap
\lbrack M,N])$ by $(-C)\cap \lbrack A,B]$ (or $(-C)\cap \lbrack M,N])\ $due
to the symmetry of $C.$
\end{remark}

\begin{proof}
By the construction of $G_{n}$ ($G_{n}^{\prime }$), i.e. $G_{n+1}\subset
G_{n}$ ($G_{n+1}^{\prime }\subset G_{n}^{\prime }$) for any $n\geq k_{0}$,
it follows that
\begin{equation*}
C\cap \lbrack A,B]=\cap _{n=k_{0}}^{\infty }G_{n}\text{ and }C\cap \lbrack
M,N]=\cap _{n=k_{0}}^{\infty }G_{n}^{\prime }.
\end{equation*}%
The continuity of $F$ yields that
\begin{equation*}
F(C\cap \lbrack A,B],C\cap \lbrack M,N])=\cap _{n=k_{0}}^{\infty
}F(G_{n},G_{n}^{\prime }).
\end{equation*}%
In terms of the relation $G_{n+1}=\widetilde{G_{n}}$, $G_{n+1}^{\prime }=%
\widetilde{G_{n}^{\prime }}$ and the condition in the lemma, it follows that
\begin{eqnarray*}
F(G_{n},G_{n}^{\prime }) &=&\cup _{1\leq i\leq t_{n}}\cup _{1\leq j\leq
t_{n}^{\prime }}F(I_{n,i},J_{n,j}) \\
&=&\cup _{1\leq i\leq t_{n}}\cup _{1\leq j\leq t_{n}^{\prime }}F(\widetilde{%
I_{n,i}},\widetilde{J_{n,j}}) \\
&=&F(\cup _{1\leq i\leq t_{n}}\widetilde{I_{n,i}},\cup _{1\leq j\leq
t_{n}^{\prime }}\widetilde{J_{n,j}}) \\
&=&F(G_{n+1},G_{n+1}^{\prime }).
\end{eqnarray*}%
Therefore, $F(C\cap \lbrack A,B],C\cap \lbrack
M,N])=F(G_{k_{0}},G_{k_{0}}^{\prime }).$
\end{proof}

\begin{proposition}
\label{key3} Assume that $f:U\rightarrow \mathbb{R}$ is a function such that
$\partial _{x}f$, $\partial _{y}f$ are continuous on $U.$ If there is a
point $(x_{0},y_{0})\in (C\times C)\cap U$ such that
\begin{equation*}
\left\{
\begin{array}{l}
\partial _{x}f|_{(x_{0},y_{0})},\partial _{y}f|_{(x_{0},y_{0})}>0, \\
\partial _{y}f|_{(x_{0},y_{0})}<\partial _{x}f|_{(x_{0},y_{0})}<3\partial
_{y}f|_{(x_{0},y_{0})},%
\end{array}%
\right.
\end{equation*}%
then $f_{U}(C,C)$ contains a non-empty interior.
\end{proposition}

\begin{proof}
Let $E$\ be a basic square\ of $C\times C$ containing $(x_{0},y_{0})$ such
that the diameter of $E$ is small enough. By Lemma \ref{key1}, it suffices
to show that for any basic square $I\times J=[a_{1},a_{1}+3t]\times \lbrack
a_{1},a_{1}+3t]\subset E,$ we have
\begin{equation*}
f(\widetilde{I},\widetilde{J})=\bigcup%
\limits_{i,j=1}^{2}f(I_{i},J_{j})=f(I,J),
\end{equation*}%
where $\widetilde{I}=I_{1}\cup I_{2}=[a_{1},a_{1}+t]\cup \lbrack
a_{2},a_{2}+t]$ and $\widetilde{J}=J_{1}\cup J_{2}=[b_{1},b_{1}+t]\cup
\lbrack b_{2},b_{2}+t].$

\begin{figure}[tbph]
\centering\includegraphics[width=0.35\textwidth]{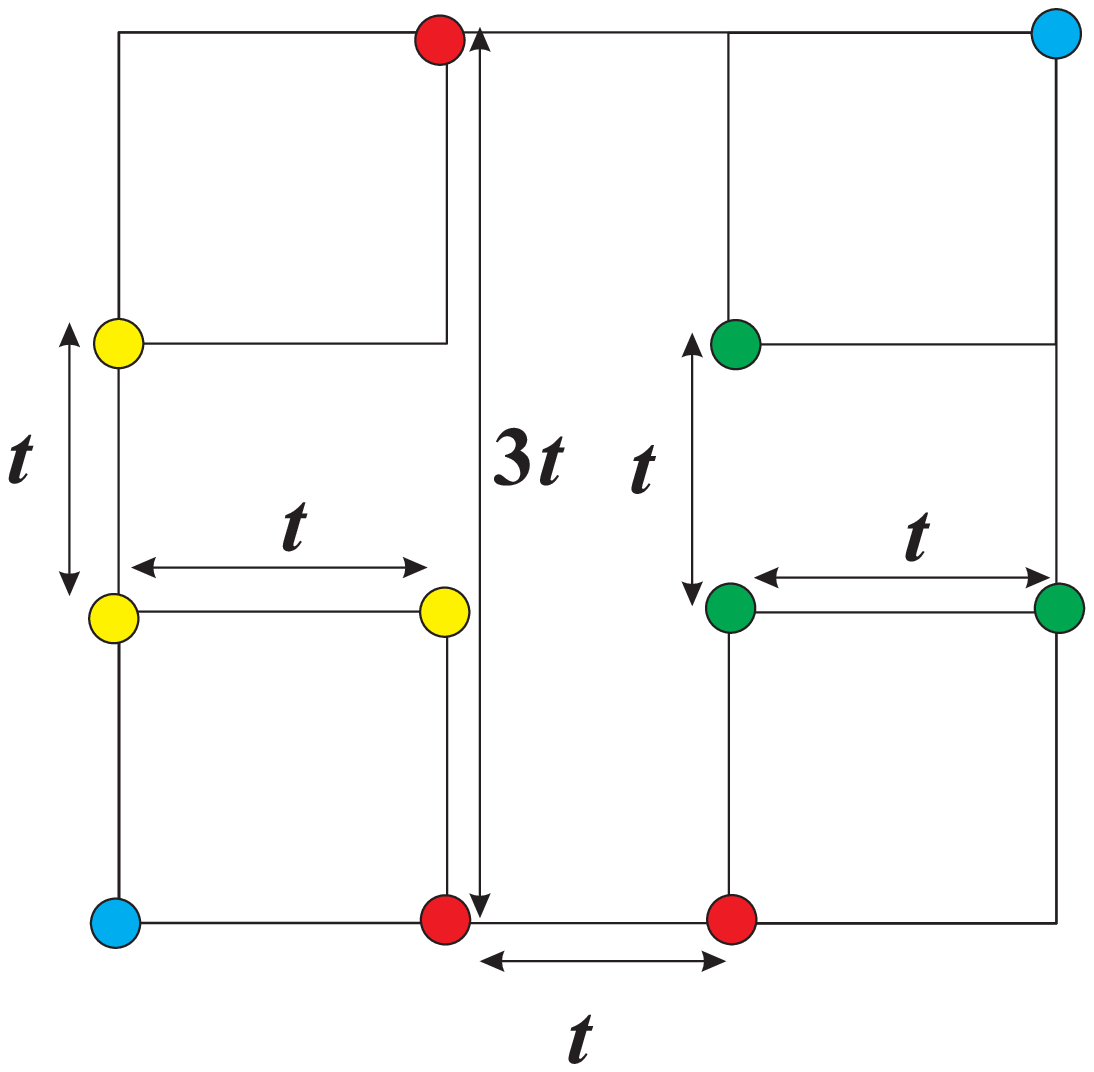} \vspace{-0.2cm}
\caption{{}}
\end{figure}

Note that $f$ is differentiable since $\partial _{x}f$, $\partial _{y}f$ are
continuous. Since $t$ is small enough, when $(x,y)\rightarrow (x_{0},y_{0}),$
by the above conditions we obtain that
\begin{eqnarray*}
H_{1}(x,y) &=&f(x+3t,y+3t)-f(x,y)=3(\partial _{x}f+\partial _{y}f)t+o(t)>0,
\\
H_{2}(x,y) &=&f(x+t,y+t)-f(x,y)=(\partial _{x}f+\partial _{y}f)t+o(t)>0, \\
H_{3}(x,y) &=&f(x,y+3t)-f(x+t,y)=(3\partial _{y}f-\partial _{x}f)t+o(t)>0, \\
H_{4}(x,y) &=&f(x+t,y)-f(x,y+t)=(\partial _{x}f-\partial _{y}f)t+o(t)>0,
\end{eqnarray*}%
where $o(t)/t\rightarrow 0$ uniformly as $t\rightarrow 0$, i.e. $o(t)$ is
independent of the choice of $(x,y)$ as $\partial _{x}f$ and $\partial _{y}f$
are continuous. Using $\partial _{x}f,\partial _{y}f>0$, $H_{1}(x,y)>0$ and $%
H_{2}(x,y)>0$ we have
\begin{equation*}
f(I,J)=[f(a_{1},b_{1}),f(a_{1}+3t,b_{1}+3t)],
\end{equation*}%
\begin{equation*}
f(I_{i},J_{j})=[f(a_{i},b_{j}),f(a_{i}+t,b_{j}+t)]\text{ for all }i,j,
\end{equation*}%
and
\begin{eqnarray*}
f(I_{1},J_{1})\cap f(I_{1},J_{2}) &\neq &\emptyset \text{ since }%
f(a_{1}+t,b_{1}+t)-f(a_{1},b_{2})=H_{4}(a_{1},b_{1}+t)>0, \\
f(I_{1},J_{2})\cap f(I_{2},J_{1}) &\neq &\emptyset \text{ since }%
f(a_{1}+t,b_{2}+t)-f(a_{2},b_{1})=H_{3}(a_{1}+t,b_{1})>0, \\
f(I_{2},J_{1})\cap f(I_{2},J_{2}) &\neq &\emptyset \text{ since }%
f(a_{2}+t,b_{1}+t)-f(a_{2},b_{2})=H_{4}(a_{2},b_{1}+t)>0.
\end{eqnarray*}%
Therefore we obtain $f(\widetilde{I},\widetilde{J})=f(I,J).$ The proposition
follows from Lemma \ref{key1}.
\end{proof}

\medskip

\begin{proof}[\textbf{Proof of Theorem 1}]
$\ $

(1) Case 1: Suppose $\partial _{x}f|_{(x_{0},y_{0})}>0$ and $\partial
_{y}f|_{(x_{0},y_{0})}>0.$ If
\begin{equation*}
\partial _{x}f|_{(x_{0},y_{0})}<\partial _{y}f|_{(x_{0},y_{0})}<3\partial
_{x}f|_{(x_{0},y_{0})},
\end{equation*}
we replace $f(x,y)$ by $g_{1}(x,y)=f(y,x),$ then Theorem 1 follows from
Proposition \ref{key3} in this case.

(2) Case 2: Suppose $\partial _{x}f|_{(x_{0},y_{0})}<0$ and $\partial
_{y}f|_{(x_{0},y_{0})}<0.$ In this case, we can replace $f(x,y)$ by $%
g_{2}(x,y)=-f(x,y)$ or $g_{3}(x,y)=-f(y,x).$

(3) Case 3: Suppose $\partial _{x}f|_{(x_{0},y_{0})}<0$ and $\partial
_{y}f|_{(x_{0},y_{0})}>0.$ In this case, we can replace $f(x,y)$ by $%
g_{4}(x,y)=f(-x,y),$ we obtain $\partial _{x}g_{4}|_{(x_{0},y_{0})}>0$ and $%
\partial _{y}g_{4}|_{(x_{0},y_{0})}>0.$ By the symmetry of $C$ and $(-C),$
applying Lemma \ref{key1} to $(-C)\times C,$ Theorem \ref{Main} follows from
Proposition \ref{key3}.

(4) Case 4: Suppose $\partial _{x}f|_{(x_{0},y_{0})}>0$ and $\partial
_{y}f|_{(x_{0},y_{0})}<0.$ We can replace $f(x,y)$ by $g_{5}(x,y)=f(x,-y)$
in this case$.$
\end{proof}

\medskip

\begin{proof}[\textbf{Proof of Corollary \protect\ref{Cor}}]
It suffices to check the conditions in Theorem \ref{Main}.

(1) If $f(x,y)=x^{\alpha }y^{\beta }$ with $\alpha \beta \neq 0$, using $%
(x^{\alpha }y^{\beta })=(xy^{\beta /\alpha })^{\alpha },$ we only need to
deal with $f(x,y)=xy^{\gamma }.$ Using the symmetry, we may assume that $%
|\gamma |\geq 1.$ Now we have%
\begin{equation*}
\partial _{x}f=y^{\gamma }\text{ and }\partial _{y}f=\gamma xy^{\gamma -1}
\end{equation*}%
with $\left\vert \frac{\partial _{x}f}{\partial _{y}f}\right\vert =\frac{1}{%
|\gamma |}\left\vert \frac{y}{x}\right\vert .$ If $|\gamma |=3^{k}$ for some
integer $k\geq 0,$ we take $y=1$ and $x=(2/3)\cdot 3^{-k},$ hence $%
\left\vert \frac{\partial _{x}f}{\partial _{y}f}\right\vert =3/2\in (1,3)$
in this case$.$ Otherwise, if $3^{k}<|\gamma |<3^{k+1}\ $for some integer $%
k\geq 0,$ then we take $y=1$ and $x=3^{-(k+1)},$ then $\left\vert \frac{%
\partial _{x}f}{\partial _{y}f}\right\vert \in (1,3).$ Now, $f_{U}(C,C)$
contains a non-empty interior for $f(x,y)=x^{\alpha }y^{\beta }$ with $%
\alpha \beta \neq 0$.

(2) If $f(x,y)=x^{\alpha }\pm y^{\alpha }$ with $\alpha \neq 0$, then%
\begin{equation*}
|\partial _{x}f|=|\alpha |x^{\alpha -1}\text{ and }|\partial _{y}f|=|\alpha
|y^{\alpha -1}
\end{equation*}%
with $\left\vert \frac{\partial _{x}f}{\partial _{y}f}\right\vert
=\left\vert \frac{x^{\alpha -1}}{y^{\alpha -1}}\right\vert .$ When $\alpha
\neq 1,$ take $x,y\in C$ such that $y/x$ is close to $1$ enough$,$ then $%
1<\left\vert \frac{\partial _{x}f}{\partial _{y}f}\right\vert <3$ or $%
1<\left\vert \frac{\partial _{y}f}{\partial _{x}f}\right\vert <3.$ When $%
\alpha =1,$ the classical result $C+C=[0,2]\ $implies there is a non-empty
interior in $f(C,C)=C+C.$

(3) If $f(x,y)=x\pm y^{2}$, then $\partial _{x}f=1,|\partial _{y}f|=2y$.
Take $x_{0}=8/9,y_{0}=1/3$, which implies $1<|1/(2y_{0})|<3$.

(4) If $f(x,y)=\sin (x)\cos (y),$ then
\begin{equation*}
|\partial _{x}f|=|\cos x\cos y|,|\partial _{y}f|=|\sin x\sin y|.
\end{equation*}%
We take $(x_{0},y_{0})=(2/3,2/3),$ we obtain that
\begin{equation*}
|\cos (2/3)\cos (2/3)|=0.6176\cdots ,\text{ }|\sin (2/3)\sin
(2/3)|=0.3823\cdots ,
\end{equation*}%
and thus $1<\left\vert \frac{\cos (2/3)\cos (2/3)}{\sin (2/3)\sin (2/3)}%
\right\vert =1.615\cdots <3.$
\end{proof}

\bigskip

\section{Final remarks}

Our idea can be implemented for some overlapping self-similar sets.
Moreover, in Theorem \ref{Main}, for some functions, we can obtain that $%
f_{U}(C,C)$ contains infinitely many closed intervals.

\bigskip


\end{document}